\documentclass{amsart}
\usepackage{amsmath,amssymb,amsthm,amsfonts,enumerate,array,color,lscape,fancyhdr,layout,pst-all}
\usepackage[english]{babel}
\usepackage[latin1]{inputenc}
\usepackage{young}

\usepackage[hcentermath]{youngtab}
\usepackage{graphicx}
\usepackage{float}
\usepackage{pstricks}
\usepackage[all]{xy}

\newcounter{numberofremark}
\newcommand\nothing[1]{}
\usepackage[active]{srcltx}
\usepackage[hyperindex,bookmarksnumbered,plainpages]{hyperref}

\newcommand{\dcl}{\DeclareMathOperator}
\dcl\cdet{cdet} \dcl\Sp{Specm} \dcl\depth{depth} \dcl\im{Im} \dcl\Span{span} \dcl\Ker{Ker} \dcl\Specm{Specm}
\dcl\Supp{Supp} \dcl\codim{codim} \dcl\Y{Y} \dcl\gl{\mathfrak{gl}}    \dcl\U{U} \dcl\T{T}
\dcl\qdet{qdet} \dcl\sgn{sgn} \dcl\gr{gr} \dcl\diag{diag}
\dcl\g{\mathfrak{g}} \dcl\C{\mathbb C} \dcl\dd{{\mathrm d}}

\newcommand\sm{{\mathsf m}}

\newcommand\Ga{{\Gamma}}

\newlength\yStones
\newlength\xStones
\newlength\xxStones

\makeatletter
\def\Stones{\pst@object{Stones}}
\def\Stones@i#1{%
  \pst@killglue%
  \begingroup%
  \use@par%
  \setlength\xxStones{\xStones}%
  \expandafter\Stones@ii#1,,\@nil
  \endgroup
  \global\addtolength\xStones{0.6cm}%
  \global\addtolength\yStones{-7.5mm}}%
\def\Stones@ii#1,#2,#3\@nil{%
  \rput(\xxStones,\yStones){%
    \psframebox[framesep=0]{%
      \parbox[c][6mm][c]{11mm}{\makebox[11mm]{$#1$}}}}%
  \addtolength\xxStones{1.2cm}%
  \ifx\relax#2\relax\else\Stones@ii#2,#3\@nil\fi}
\makeatother

\def\Stone#1{\fbox{\makebox[10mm]{\strut#1}}\kern2pt}

\newtheorem{theorem}{Theorem}[section]
\newtheorem{lemma}[theorem]{Lemma}
\newtheorem{corollary}[theorem]{Corollary}
\newtheorem{proposition}[theorem]{Proposition}
\newtheorem{example}[theorem]{Example}
\newtheorem{remark}[theorem]{Remark}

\newtheorem{definition}[theorem]{Definition}

\begin{document}
\title{Explicit construction of  irreducible modules for $U_q(\mathfrak{gl}_n)$}
\author{Vyacheslav Futorny}
\address{Instituto de Matem\'atica e Estat\'istica, Universidade de S\~ao
Paulo,  S\~ao Paulo SP, Brasil} \email{futorny@ime.usp.br}
\author{Luis Enrique Ramirez}
\address{Centro de Matem\'atica, Computa\c c\~ao e Cogni\c c\~ao, Universidade Federal do ABC (AFABC), Santo Andr\'e SP, Brasil} \email{luis.enrique@ufabc.edu.br}
\author{Jian Zhang}
\address{Instituto de Matem\'atica e Estat\'istica, Universidade de S\~ao
Paulo,  S\~ao Paulo SP, Brasil} \email{zhang@ime.usp.br}

\begin{abstract}
We construct new families of $U_q(\mathfrak{gl}_n)$-modules by continuation from finite dimensional representations.
Each such module is
associated with a combinatorial object -  admissible set of relations defined in \cite{FRZ}. More precisely, we prove that any admissible set of relations leads to a family of irreducible  $U_q(\mathfrak{gl}_n)$-modules.  Finite dimensional and generic modules are  particular cases of this construction.

\end{abstract}


\maketitle
\section{Introduction}

Gelfand and Graev \cite{GG} proposed a method of constructing  of $\gl_n$-modules which extend finite dimensional modules and admit a basis of tableaux with the standard action of the generators of the Lie algebra \cite{GT}.    This construction is based on a choice of certain relations satisfied by the entries of the Gelfand-Tsetlin tableaux.      Lemire and Patera \cite{LP} conjectured sufficient conditions under which the Gelfand-Graev defines in fact a module, proving it for $n=3$ and $n=4$. In \cite{FRZ} the authors proved this conjecture and extended
the construction for a larger class of irreducible $\gl_n$-modules.  The purpose of this letter is to show how to deform this construction and obtain new large  families of irreducible modules for the quantum group  $U_q(\mathfrak{gl}_n)$.  Infinite dimensional generic and finite dimensional modules are particular cases of this construction.
 New irreducible modules are presented explicitly with a basis consisting of certain tableaux and with explicit action of the generators of
$U_q(\mathfrak{gl}_n)$ generalizing  the construction of finite dimensional representations \cite{UST2}. Having such an explicit construction will be useful for possible applications.

Constructed modules belong to the category of Gelfand-Tsetlin modules with a diagonalizable action of the Gelfand-Tsetlin subalgebra.
For $\gl_n$ the theory of Gelfand-Tsetlin modules has origin in the classical paper of Gelfand and Tsetlin \cite{GT}.   It is related to many concepts arizing in Mathematics and Physics, see for example \cite{KW1}, \cite{KW2},  \cite{GS}, \cite{FM}, \cite{Gr1}, \cite{Gr2}, \cite{CE1}, \cite{CE2}, \cite{FO1}.
 The general theory of Gelfand-Tsetlin modules for $\gl_n$
was developed in \cite{DFO}, \cite{O}, \cite{FO2}, \cite{Maz1}, \cite{Maz2}, \cite{m:gtsb}, \cite{FGR1}, \cite{FGR2}, \cite{FGR3} and references therein.
For $U_q(\mathfrak{gl}_n)$ certain families of  Gelfand-Tsetlin modules were constructed in \cite{MT} and \cite{FRZ1}, while the general theory  was developed in \cite{FHR}.

Current letter provides new information about  Gelfand-Tsetlin modules for\\ $U_q(\mathfrak{gl}_n)$.
The paper is organized as follows. Section 2 contains some preliminary information. In Section 3 we introduce our main technical tools - admissible sets of relations and realizable sets of relations. To any realizable set of relations
  we associate a family of $U_q(\mathfrak{gl}_n)$-modules. We prove the main result of this letter stating that any admissible set of relations is a realizable set of relations (Theorem 3.9).
  A certain effective method (RR-method) of constructing the admissible relation is described in Theorem 3.11.
  Finally, in Section 4 we study the action of the generators of the Gelfand-Tsetlin subalgebra on modules associated with admissible sets of relations. The Gelfand-Tsetlin subalgebra $\Gamma_q$ is diagonalizable on all constructed modules (Theorem 4.1), moreover, it separates the basis tableaux  (Proposition 4.4). Using the action of $\Gamma_q$ we obtain a criterion of irreducibility
  of constructed admissible modules: irreducible modules correspond to maximal sets of admissible relations (Theorem 4.5).

We now fix some notation and conventions. Throughout the paper we fix an integer $n\geq 2$ and $q\in\mathbb{C}$ which is not root of unity. The ground field will be ${\mathbb C}$.   By $U_q$ we denote the quantum enveloping algebra of $\gl_n$. We fix the standard Cartan subalgebra  $\mathfrak h$, the standard triangular decomposition and the corresponding basis of simple roots $\alpha_1, \ldots, \alpha_{n-1}$.  The weights of $U_q$ will be written as $n$-tuples $(\lambda_1, \ldots, \lambda_n)$. For a commutative ring $R$, by ${\rm Specm}\, R$ we denote the set of maximal ideals of $R$. For $i>0$ by $S_i$ we denote the $i$th symmetric group. Let $1(q)$ be the set of all complex $x$ such that $q^{x}=1$. Finally, for any complex number $x$, we set
\begin{align*}
(x)_q=\frac{q^x-1}{q-1},\quad
[x]_q=\frac{q^x-q^{-x}}{q-q^{-1}}.
\end{align*}

\

\noindent{\bf Acknowledgements.}  V.F. is
supported in part by  CNPq  (301320/2013-6) and by
Fapesp  (2014/09310-5).
J. Z. is supported by  Fapesp  (2015/05927-0).

\section{Preliminaries}

We define $U_{q}$ as a unital associative algebra generated by $e_{i}, f_{i}(1\leq i
\leq n)$ and $q^{h}(h\in \mathfrak{h})$ with the
following relations:
\begin{align}
q^{0}=1,\  q^{h}q^{h'}=q^{h+h'} \quad (h,h' \in \mathfrak{h}),\\ \label{re1}
q^{h}e_{i}q^{-h}=q^{\langle h,\alpha_i\rangle}e_{i}  ,\\
q^{h}f_{i}q^{-h}=q^{-\langle h,\alpha_i\rangle}f_{i} ,\\
e_{i}f_{j}-f_{j}e_{i}=\delta_{ij}\frac{q^{\alpha_i}-q^{-\alpha_i}}{q-q^{-1}} ,\\
e_{i}^2e_{j}-(q+q^{-1})e_ie_je_i+e_je_{i}^2=0  \quad (|i-j|=1),\\
f_{i}^2f_{j}-(q+q^{-1})f_if_jf_i+f_jf_{i}^2=0  \quad (|i-j|=1),\\
e_{i}e_j=e_je_i,\  f_if_j=f_jf_i  \quad (|i-j|>1). \label{re2}
\end{align}
The quantum special linear algebra $U_q(sl_n)$ is the subalgebra of $U_q$
generated by $e_i,\ f_i,\ q^{\pm \alpha_i}(i=1,2,\ldots, n-1)$.

We denote by $(U_m)_q$ the quantum universal enveloping algebra of $\gl_m$. We have the following chain $(U_1)_q\subset$ $(U_2)_q\subset$ $\ldots$ $\subset
(U_n)_q$. Let $Z_{m}$ denotes the center of $(U_{m})_{q}$. The subalgebra of $U_q$ generated by $\{
Z_m\,|\,m=1,\ldots, n \}$ will be called the \emph{Gelfand-Tsetlin
subalgebra} of $U_q$ and will be denoted by  ${\Ga}_q$ \cite{FHR}, \cite{FRZ1}.

%

\begin{definition}
\label{definition-of-GZ-modules} A finitely generated $U_q$-module
$M$ is called a \emph{Gelfand-Tsetlin module (with respect to
$\Ga_q$)} if

\begin{equation}\label{equation-Gelfand-Tsetlin-module-def}
M=\bigoplus_{\sm\in\Sp\Ga_q}M(\sm),
\end{equation}

where $M(\sm)=\{v\in M| \sm^{k}v=0 \text{ for some }k\geq 0\}$.
\end{definition}

 For a vector $L=(l_{ij})$ in $\mathbb{C}^{\frac{n(n+1)}{2}}$, by $T(L)$ we will denote the following array with entries $\{l_{ij}:1\leq j\leq i\leq n\}$
\begin{center}

\Stone{\mbox{ \scriptsize {$l_{n1}$}}}\Stone{\mbox{ \scriptsize {$l_{n2}$}}}\hspace{1cm} $\cdots$ \hspace{1cm} \Stone{\mbox{ \scriptsize {$l_{n,n-1}$}}}\Stone{\mbox{ \scriptsize {$l_{nn}$}}}\\[0.2pt]
\Stone{\mbox{ \scriptsize {$l_{n-1,1}$}}}\hspace{1.5cm} $\cdots$ \hspace{1.5cm} \Stone{\mbox{ \tiny {$l_{n-1,n-1}$}}}\\[0.3cm]
\hspace{0.2cm}$\cdots$ \hspace{0.8cm} $\cdots$ \hspace{0.8cm} $\cdots$\\[0.3cm]
\Stone{\mbox{ \scriptsize {$l_{21}$}}}\Stone{\mbox{ \scriptsize {$l_{22}$}}}\\[0.2pt]
\Stone{\mbox{ \scriptsize {$l_{11}$}}}\\
\medskip
\end{center}
such an array will be called a \emph{Gelfand-Tsetlin tableau} of height $n$. For any $1\leq j\leq i\leq n-1$, the vector $\delta^{ij} \in {\mathbb Z}^{\frac{n(n+1)}{2}}$ is defined by  $(\delta^{ij})_{ij}=1$ and all other $(\delta^{ij})_{k\ell}$ are zero. Finally, a Gelfand-Tsetlin tableau of height $n$ is called \emph{standard} if
$l_{ki}-l_{k-1,i}\in\mathbb{Z}_{\geq 0}$ and $l_{k-1,i}-l_{k,i+1}\in\mathbb{Z}_{>0}$ for all $1\leq i\leq k\leq n$.

Recall the quantum version of the classical result of  Gelfand and  Tsetlin which provides an explicit basis  in the finite dimensional case.

\begin{theorem}[\cite{UST2}, Theorem 2.11 and \cite{FRZ1}, Proposition 4.3]\label{Theorem: quantum GT theorem}
Let $L(\lambda)$ be the finite dimensional irreducible module over $U_q$ of highest weight $\lambda=(\lambda_{1},\ldots, \lambda_{n})$, where $\lambda_i-\lambda_{i+1}\in \mathbb{Z}_{\geq 0}$. Then there exist a basis of $L(\lambda)$ consisting of all standard tableaux $T(L)$ with fixed top row $l_{nj}=\lambda_j-j$. Moreover,  the action of the generators of $U_q$ on $L(\lambda)$ is given by the  \emph{Gelfand-Tsetlin formulae}:
\begin{equation}\label{Gelfand-Tsetlin formulas}
\begin{split}
q^{\epsilon_{k}}(T(L))&=q^{a_k}T(L),\quad a_k=\sum_{i=1}^{k}l_{k,i}-\sum_{i=1}^{k-1}l_{k-1,i}+k,\ k=1,\ldots,n,\\
e_{k}(T(L))&=-\sum_{j=1}^{k}
\frac{\prod_{i} [l_{k+ 1,i}-l_{k,j}]_q}{\prod_{i\neq j} [l_{k,i}-l_{k,j}]_q}
T(L+\delta^{kj}),\\
f_{k}(T(L))&=\sum_{j=1}^{k}\frac{\prod_{i} [l_{k-1,i}-l_{k,j}]_q}{\prod_{i\neq j} [l_{k,i}-l_{k,j}]_q}T(L-\delta^{kj}).\\
\end{split}
\end{equation}
Moreover, the generators $c_{mk}$ of $\Gamma_q$ acts on $T(L)$ as multiplication by
\begin{equation}\label{eigenvalues of gamma_mk}
\gamma_{mk}(L)=(k)_{q^{-2}}!(m-k)_{q^{-2}}!q^{k(k+1)+\frac{m(m-3)}{2}}\sum_{\tau}
q^{\sum_{i=1}^{k}l_{m\tau(i)}-\sum_{i=k+1}^{m}l_{m\tau(i)}}
\end{equation}
where $\tau\in S_m$ is such that $\tau(1)<\cdots<\tau(k),\ \tau(k+1)<\cdots<\tau(m)$.
\end{theorem}

\section{Admissible relations}


Set $\mathfrak{V}:=\{(i,j)\ |\ 1\leq j\leq i\leq n\}$. We will consider relations between elements of $\mathfrak{V}$ of the form  $(i,j)\geq (s,t)$ or $(i,j)>(s,t)$. More precisely, we will consider subsets of the following set of relations:

\begin{align}
\mathcal{R} &:=\mathcal{R}^{\geq} \cup \mathcal{R}^{>}\cup \mathcal{R}^{0},
\end{align}
where
\begin{align}
\mathcal{R^{\geq}} &:=\{(i,j)\geq(i-1,j')\ |\ 1\leq j\leq i\leq n,\ 1\leq j'\leq i-1\},\\
\mathcal{R^{>}} &:=\{(i-1,j')>(i,j)\ |\ 1\leq j\leq i\leq n,\ 1\leq j'\leq i-1\},\\
\mathcal{R}^{0} &:=\{(n,i)\geq(n,j)\ |\ 1\leq i\neq j\leq n\}.
\end{align}


Let $\mathcal{C}$ be a subset of $\mathcal{R}$.
Denote by $\mathfrak{V}(\mathcal{C})$ the set of all $(i,j)$ in $\mathfrak{V}$ such that $(i,j)\geq (\text{respectively }>,\ \leq,\ <)  \ (r, s)\in\mathcal{C}$ for some $(r, s)$.
Let $\mathcal{C}_1$ and $\mathcal{C}_2$ be two subsets of $\mathcal{C}$.
We say that $\mathcal{C}_1$ and $\mathcal{C}_2$ are {\it disconnected}
if
$\mathfrak{V}(\mathcal{C}_1)\cap\mathfrak{V}(\mathcal{C}_2)=\emptyset$,
 otherwise  $\mathcal{C}_1$ and $\mathcal{C}_2$ are connected.
 $\mathcal{C}$ is called {\it decomposable} if it can be decomposed into the  union of two disconnected subsets of $\mathcal{R}$, otherwise $\mathcal{C}$ is called indecomposable.

Note that any subset of $\mathcal{R}$ is  a union of disconnected  indecomposable sets and such decomposition is unique.

\begin{definition}
Let $\mathcal{C}$ be any subset of $\mathcal{R}$. Given $(i,j),\ (r,s)\in \mathfrak{V}(\mathcal{C})$ we will write:
\begin{itemize}
\item[(i)] $(i,j)\succeq_{\mathcal{C}} (r,s)$ if, there exists $\{(i_{1},j_{1}),\ldots,(i_{m},j_{m})\}\subseteq \mathfrak{V}(\mathcal{C})$ such that
\begin{align}\label{geq sub C}
\{(i,j)\geq (i_{1},j_1),\  (i_{1},j_1)\geq (i_2,j_2), \cdots,\  (i_{m},j_m)\geq (r,s)\}&\subseteq \mathcal{C}
\end{align}
\item[(ii)] We write $(i,j)\succ_{\mathcal{C}} (r,s)$ if there exists $\{(i_{1},j_{1}),\ldots,(i_{m},j_{m})\}\subseteq \mathfrak{V}(\mathcal{C})$ such that in the condition \ref{geq sub C}, at least one of the inequalities is $>$.
\end{itemize}
Given another set of relations $\mathcal{C}'$, we say that {\it\bf  $\mathcal{C}$   implies $\mathcal{C}'$} if whenever we have $(i,j)\succ_{\mathcal{C}'} (r,s)$ (respectively $(i,j)\succeq_{\mathcal{C}'} (r,s)$) we also have $(i,j)\succ_{\mathcal{C}} (r,s)$ (respectively $(i,j)\succeq_{\mathcal{C}} (r,s)$).
\end{definition}

\begin{definition}
Let $\mathcal{C}$ be an indecomposable subset of $\mathcal{R}$. A subset of $\mathcal{C}$ of the form $\{(k,i)\geq (k-1,t),\ (k-1,s)>(k,j)\}$ with $i<j$ and $s<t$ will be called a \emph{\bf cross}.
\end{definition}

We will define now our main concept which is a slight modification of the definition of an admissible set in \cite{FRZ}.

\begin{definition}
Let $\mathcal{C}$ be an indecomposable set. We say that $\mathcal{C} $ is {\it \bf admissible}  if it satisfies the following conditions:
\begin{itemize}
\item[(i)] For any $1\leq k-1\leq n$, we have $(k, i)\succ_{\mathcal{C}}(k, j)$ only if $i<j$;
\item[(ii)]  $(n,i)\succeq_{\mathcal{C}} (n,j)$ only if $i<j$;
\item[(iii)] There is not cross in $\mathcal{C} $;
\item[(iv)] For every $(k,i)\ (k,j)\in\mathfrak{V}(\mathcal{C})$ with  $1\leq k\leq n-1$ there exists $s< t$ such that one of the following holds
\begin{equation}\label{condition for admissible}
\begin{split}
 &\{(k,i)>(k+1,s)\geq (k,j),\ (k,i)\geq (k-1,t)>(k,j)\}\subseteq \mathcal{C},\\
 &\{(k,i)>(k+1,s),(k+1,t)\geq (k,j)\}\subseteq  \mathcal{C}.
 \end{split}
\end{equation}
\end{itemize}
An arbitrary set $\mathcal{C}$ is admissible if every  indecomposable subset of $\mathcal{C}$ is admissible.
\end{definition}

Denote by $\mathfrak{F}$ the set of all indecomposable admissible subsets. Note that we obtain the same admissible sets as in  \cite{FRZ}.
In the following we will use the admissible sets to construct modules for $U_q$.


\subsection{Tableaux realization of admissible sets of relations}
In this section we will describe $\mathbb{C}$-vector spaces associated with sets of relations $\mathcal{C}$ with Gelfand-Tsetlin tableaux as a bases. We will prove that we have a structure of  a $U_{q}$-module on such space with the action of the generators of $U_q$ given by the Gelfand-Tsetlin formulas (\ref{Gelfand-Tsetlin formulas}).

%

 \begin{definition} Let $\mathcal{C}$ be any subset of $\mathcal{R}$ and $T(L)$ any Gelfand-Tsetlin tableau. Recall that $1(q):=\{x\in \mathbb{C}\ |\ q^x=1\}$.
 \begin{itemize}
\item[1.]
\begin{itemize}
\item[(i)]
We say that $T(L)$ satisfies a relation $(i,j)\geq (r,s)$ (respectively, $(i,j)> (r,s)$) if $l_{ij}-l_{st}\in \mathbb{Z}_{\geq 0}+\frac{1(q)}{2}\ (\text{respectively, }l_{ij}-l_{st}\in \mathbb{Z}_{> 0}+\frac{1(q)}{2})$.
\item[(ii)]
We say that a Gelfand-Tsetlin tableau $T(L)$ \emph{satisfies $\mathcal{C}$} if $T(L)$ satisfies all the relations in $\mathcal{C}$
and $l_{ki}-l_{kj}\in \mathbb{Z}+\frac{1(q)}{2}$ only if $(k,i)$ and $(k,j)$
in the same indecomposable subset of $\mathfrak{V}(\mathcal{C})$. In this case we call $T(L)$ a $\mathcal{C}$-\emph{\bf realization}.
\item[(iii)] $\mathcal{C}$ is a \emph{\bf maximal} set of relations for $T(L)$ if $T(L)$ satisfies $\mathcal{C}$ and whenever $T(L)$ satisfies a set of relations $\mathcal{C}'$ we have that $\mathcal{C}$ implies $\mathcal{C}'$.
\end{itemize}
\item[2.] If $T(L)$  satisfies  $\mathcal{C}$ we denote by ${\mathcal B}_{\mathcal{C}}(T(L))$  the set of all tableaux of the form $T(L+z)$ ($z\in \mathbb Z^{\frac{n(n-1)}{2}}$) that satisfy $\mathcal{C}$, and by $V_{\mathcal{C}}(T(L))$ the complex vector space spanned by ${\mathcal B}_{\mathcal{C}}(T(L))$.

\end{itemize}
\end{definition}

\begin{example} Set
\begin{align}\label{def of S}
\mathcal{S} &:=\{(i+1,j)\geq(i,j)>(i+1,j+1)\ |\ 1\leq j\leq i\leq n-1\}.
\end{align}
\noindent
 It follows from the definition that $\emptyset$ and $\mathcal{S}$ are admissible sets of relations. Moreover, the set of all tableaux satisfying $\mathcal{S}$ coincides with the set of all standard tableaux and the set of all tableaux satisfying $\emptyset$ coincide with the set of all generic tableaux.
\end{example}

Our goal is to show that any admissible set of relations $\mathcal{C}$ leads to a family of $U_{q}$-modules. In fact, for any $\mathcal{C}$-realization $T(L)$ we will prove that $V_{\mathcal{C}}(T(L))$ is a Gelfand-Tsetlin module with the action of the generators of $U_{q}$ given by the Gelfand-Tsetlin formulas (\ref{Gelfand-Tsetlin formulas}). For this we will need some technical lemmas.


Set
\begin{equation}
e_{ki}(L)=\left\{
\begin{array}{cc}
0,& \text{ if } T(L)\notin  \mathcal {B}_{\mathcal{C}}(T(L))\\
-\frac{\prod_{j=1}^{k+1}[l_{ki}-l_{k+1,j}]_q}{\prod_{j\neq i}^{k}[l_{ki}-l_{kj}]_q},& \text{ if } T(L)\in  \mathcal {B}_{\mathcal{C}}(T(L))
\end{array}
\right.
\end{equation}

\begin{equation}
f_{ki}(L)=\left\{
\begin{array}{cc}
0,& \text{ if } T(L)\notin  \mathcal {B}_{\mathcal{C}}(T(L))\\
\frac{\prod_{j=1}^{k-1}[l_{ki}-l_{k-1,j}]_q}{\prod_{j\neq i}^{k}[l_{ki}-l_{kj}]_q},& \text{ if } T(L)\in  \mathcal {B}_{\mathcal{C}}(T(L))
\end{array}
\right.
\end{equation}

\begin{equation}
h_{k}(L)=\left\{
\begin{array}{cc}
0,& \text{ if } T(L)\notin  \mathcal {B}_{\mathcal{C}}(T(L))\\
q^{2\sum_{i=1}^{k}l_{ki}-\sum_{i=1}^{k-1}l_{k-1,i}-\sum_{i=1}^{k+1}l_{k+1,i} -1},& \text{ if } T(L)\in\mathcal {B}_{\mathcal{C}}(T(L))
\end{array}
\right.
\end{equation}

\begin{equation}
\Phi(L,z_1,\ldots,z_m)=
\left\{
\begin{array}{cc}
1,& \text{ if }T(L+z_1+\ldots+z_t)\in  \mathcal {B}_{\mathcal{C}}(T(L)) \text{ for any } t\\
0,& \text{ otherwise}.
\end{array}
\right.
\end{equation}

We will denote by $T(v)$ the tableau with variable entries $v_{ij}$.
\begin{lemma}\label{action of center1}
Let $\mathcal{C}\in\mathfrak{F}$, $T(L)$ any tableau satisfying $\mathcal{C}$.
\begin{itemize}
\item[(i)] If $T(L+\delta^{kj})\notin \mathcal {B}_{\mathcal{C}}(T(L))$ and
$l_{k,i}-l_{kj}\notin 1+\frac{1(q)}{2}$ for any $i$, then  $$\lim\limits_{v\rightarrow l}e_{kj}(v)f_{kj}(v+\delta^{kj})=0.$$
\item[(ii)] If $T(L-\delta^{kj})\notin \mathcal {B}_{\mathcal{C}}(T(L))$ and  $l_{k,j}-l_{k,i}\notin 1+\frac{1(q)}{2}$ for any i,
 then $$\lim\limits_{v\rightarrow l}f_{kj}(v)e_{kj}(v-\delta^{kj})=0.$$
\item[(iii)] If $l_{k,i}-l_{k,j}\notin 1+\frac{1(q)}{2}$,
 then $T(L+\delta^{k,j}), T(L-\delta^{k,i})\notin \mathcal {B}_{\mathcal{C}}(T(L))$, and
$$\lim\limits_{v\rightarrow l}e_{kj}(v)f_{kj}(v+\delta^{k,j})-f_{ki}(v)e_{ki}(v-\delta^{k,i})=0.$$
\end{itemize}
\end{lemma}

\begin{proof}Since $T(L+\delta^{kj})\notin \mathcal {B}_{\mathcal{C}}(T(L))$, we  have $\{(k+1,s)\geq (k,j)\}\subseteq\mathcal{C}$ or $\{(k-1,t)> (k,j)\}\subseteq \mathcal{C}$.
Suppose $\{(k+1,s)\geq (k,j)\}\subseteq\mathcal{C}$ and $T(L+\delta^{kj})\notin \mathcal {B}_{\mathcal{C}}(T(L))$. Then $l_{k+1,s}-l_{k,j}\notin\frac{1(q)}{2}$ and by direct computation one has  $\lim\limits_{v\rightarrow l}e_{kj}(v)f_{kj}(v+\delta^{kj})=0$.
Suppose $T(L+\delta^{kj})$ does not satisfies the relation $l_{k-1,t}-l_{k,j}\in \mathbb{Z}_{>0}+\frac{1(q)}{2}$. Then we have
$l_{k-1,t}-l_{k,j}\in 1+\frac{1(q)}{2}$ and $$\lim\limits_{v\rightarrow l}e_{kj}(v)f_{kj}(v+\delta^{kj})=0.$$

The proof of (ii) is similar to (i).\\
\noindent
It is clear that $T(L-\delta^{k,j}),\  T(L+\delta^{k,j+1})\notin \mathcal {B}_\mathcal{C}(T(L))$ if $l_{k,j}-l_{k,j+1}\in 1+\frac{1(q)}{2}$.
It is easy to see that $\#\{l_{k+1,i'},l_{k-1,j'}\ |\ l_{k+1,i'}-l_{kj}\in\frac{1(q)}{2},\ l_{k-1,j'}-l_{ki}\in\frac{1(q)}{2}\}\geq 2$.
By direct computation one has
$$\lim_{v\rightarrow l}e_{kj}(v)f_{kj}(v+\delta^{k,j+1})-f_{kj}(v)e_{kj}(v-\delta^{k,j})=0.$$
\end{proof}
\begin{lemma}[\cite{FRZ}, Lemma 4.21] \label{lemma phi1}
 Let $\mathcal{C}\in\mathfrak{F}$, $z^{(1)},z^{(2)}\in \mathbb{Z}^{\frac{n(n-1)}{2}}$. Denote $I_1=\{(i,j)\ |\ z_{ij}^{(1)}\neq 0\}$, $I_2=\{(i,j)\ |\ z_{ij}^{(2)}\neq 0\}$. If $I_1\cap I_2=\emptyset$ and for any $(i_1,j_1)\in I_1$,
    $(i_2,j_2)\in I_2$ there is no relation between $(i_1,j_1)$ and $(i_2,j_2)$, then
     $T(R+z^{(1)}+z^{(2)})\neq 0$ if and only if $T(R+z^{(1)})\neq 0$ and $T(R+z^{(2)})\neq 0$.
\end{lemma}

\

\subsection{$U_{q}$-modules defined by admissible relations}

\begin{definition}
Let $\mathcal{C}$ be a subset of $\mathcal{R}$. We call $\mathcal{C}$  \emph{\bf realizable} if for any tableau $T(L)$ satisfying  $\mathcal{C}$,  the vector space $V_{\mathcal{C}}(T(L))$  has a structure of a $U_q$-module, endowed with the action of $U_q$ given by the Gelfand-Tsetlin formulas (\ref{Gelfand-Tsetlin formulas}).
\end{definition}

\begin{theorem}\label{sufficiency of admissible} If $\mathcal{C}$ is a union of disconnected sets from $\mathfrak{F}$ then $\mathcal{C}$ is realizable.
\end{theorem}
\begin{proof} Let $\mathcal{C}$ be a  union of disconnected sets from $\mathfrak{F}$. It is sufficient to consider the case when $\mathcal{C}$ is a union of two disconnected subsets from $\mathfrak{F}$.
Suppose $\mathcal{C}=\mathcal{C}_1\cup\mathcal{C}_2$.

Let $T(L)$ be any $\mathcal{C}$-realization.
In order to prove that $V_{\mathcal{C}}(T(L))$ is a $U_{q}$-module one needs to verify all the defining relations (\ref{re1}--\ref{re2}) for any $T(R)\in {\mathcal B}_{\mathcal{C}}(T(L))$.

First we show that $(e_{i}^2e_{j}-(q+q^{-1})e_ie_je_i+e_je_{i}^2) T(R)=0 \, (|i-j|=1)$.
\begin{equation}\label{formula}
\begin{aligned}
&(e_{i}^2e_{j}-(q+q^{-1})e_ie_je_i+e_je_{i}^2)T(R)\\
=&\sum_{r,s,t}\Phi(R,\delta^{jr},\delta^{is})
e_{jr}(R)e_{is}(R+\delta^{jr})e_{it}(R+\delta^{jr}+\delta^{is})T(R+\delta^{jr}+\delta^{is}+\delta^{it})\\
+&\sum_{r,s,t}\Phi(R,\delta^{is},\delta^{it})
e_{is}(R)e_{it}(R+\delta^{is})e_{jr}(R+\delta^{is}+\delta^{it})T(R+\delta^{jr}+\delta^{is}+\delta^{it})\\
-&[q]_2\sum_{r,s,t}\Phi(R,\delta^{is},\delta^{jr})
e_{is}(R)e_{jr}(R+\delta^{is})e_{it}(R+\delta^{is}+\delta^{jr})T(R+\delta^{jr}+\delta^{is}+\delta^{it}).\\
\end{aligned}
\end{equation}

Now we consider the coefficients by nonzero tableau $T(R+\delta^{jr}+\delta^{is}+\delta^{it})$.

\item [(i)] Let $s=t$.

\begin{itemize}
\item[(a)] Suppose there is no relation between $(i,s)$ and $(j,r)$. Then $\Phi(R,\delta^{jr},\delta^{is})=\Phi(R,\delta^{is},\delta^{is})=\Phi(R,\delta^{is},\delta^{jr})=1$ by Lemma \ref{lemma phi1}.
Then the coefficient of $T(R+\delta^{jr}+2\delta^{is})$ is the limit of the coefficient of $T(v+\delta^{jr}+2\delta^{is})$ when $v\rightarrow R$ (here $T(v)$ again is a tableau with variable entries). Thus
the coefficient of $T(R+\delta^{jr}+2\delta^{is})$ is zero.
\item[(b)] Suppose there exists a relation between $(i,s)$ and $(j,r)$. Without loss of generality we assume that this relation is  $\mathcal {C'}=\{(i,s)\geq (j,r)\}$,
      Let $T(v')$ be the tableau with $v'_{s't'}=l_{s't'}$ if $(s',t')=(i,s) \text{ or } (j,r)$, and
      variable entries otherwise. Then $T(v')$ is a $\mathcal {C'}$-realization and $V_{\mathcal{C}'}(T(v'))$ is a module for arbitrary generic values of free variables in $v'$.
        Let $z^{(1)},z^{(2)}\in \{\delta^{jr},\delta^{is}\}$. Then $\Phi(R,z^{(1)},z^{(2)})=\Phi(v,z^{(1)},z^{(2)})$ where $z^{(1)}=z^{(2)}$ only if  $z^{(1)}=z^{(2)}=\delta^{is}$. Therefore the coefficient of $T(R+\delta^{jr}+2\delta^{is})$ is the limit of the coefficient of $T(v+\delta^{jr}+2\delta^{is})$ when $v\rightarrow R$, hence, it is zero.
\end{itemize}
\item [(ii)] Suppose $s\neq t$. Then there is no relation between $(i,s)$ and $(i,t)$.
\begin{itemize}
\item[(a)] Suppose there is no relation between $(j,r)$ and $(i,s)$ or between $(j,r)$ and $(i,t)$. Then
 the value of the function $\Phi$ that  appears along with $T(R+\delta^{jr}+\delta^{is} +\delta^{it})$ is $1$ by Lemma \ref{lemma phi1}.
 Thus the coefficient of $T(R+\delta^{jr}+\delta^{is} +\delta^{it})$ is zero similarly to (a) in (i).

\item[(b)] Suppose there is a relation between $(j,r)$ and one of $\{(i,s),(i,t)\}$. Similarly to (b) in (i), one has that the coefficient of $T(R+\delta^{jr}+\delta^{is} +\delta^{it})$ is zero.

\item[(c)] Suppose there exist relations between $(j,r)$ and both $\{(i,s), (i,t)\}$.
In this case $(j,r),(i,s),(i,t)$ are in the same indecomposable set.
If $r_{is}-r_{it}=1$ then $r_{jr}=r_{it}$ and there exists $r'$  such that
$\{(i,s)\geq (i-1,r')> (i,t))\} \subseteq \mathcal{C}$ and
$r_{i-1, r'}=r_{is}$. It contradicts with $T(R+\delta^{jr}+\delta^{is}+\delta^{it})$ nonzero.
Therefore $r_{is}-r_{it}\in \frac{1(q)}{2}+\mathbb{Z}_{>0}$.
 Then $r_{is}-r_{jr}\in \frac{1(q)}{2}+\mathbb{Z}_{>0}$ or $r_{jr}-r_{it}\in \frac{1(q)}{2}+\mathbb{Z}_{>0}$.
      Without loss of generality we assume that $r_{jr}-r_{it}\in \frac{1(q)}{2}+\mathbb{Z}_{>0}$.
      Let $\mathcal {C'}=\{(i,s)\geq (j,r)\}$ and $T(v')$  the tableau with $v'_{s't'}=l_{s't'}$ if $(s',t')=(i,s) \text{ or } (j,r)$ and variable entries
       otherwise. Then $T(v')$ is a $\mathcal {C'}$-realization and $V_{\mathcal{C}'}(T(v'))$ is a module.
      Let $z^{(1)},z^{(2)}\in \{\delta^{jr},\delta^{is},\delta^{it}\}$. One has that $\Phi(R,z^{(1)},z^{(2)})=\Phi(v,z^{(1)},z^{(2)})$ whenever $z^{(1)}\neq z^{(2)}$. Therefore the coefficient of $T(R+\delta^{jr}+2\delta^{is})$ is the limit of the coefficient of $T(v+\delta^{jr}+2\delta^{is})$ when $v\rightarrow R$, which  is zero.
\end{itemize}



In the following we show that $(e_{i}f_{j}-f_{j}e_{i})T(R)=\delta_{ij}\frac{q^{\alpha_i}-q^{-\alpha_i}}{q-q^{-1}}T(R)$. We have
\begin{equation}
\begin{split}
    (e_{i}f_{j}-f_{j}e_{i})T(R)
    =&\sum_{r=1}^{j}\sum_{s=1}^{i}\Phi(R,-\delta^{jr})f_{jr}(R)e_{is}(R+\delta^{jr})T(R-\delta^{jr}+\delta^{is})\\
-&\sum_{r=1}^{j}\sum_{s=1}^{i}\Phi(R,\delta^{is})e_{is}(R)f_{jr}(R+\delta^{is})T(R-\delta^{jr}+\delta^{is}).\\
\end{split}
\end{equation}

Now we consider the coefficients of nonzero tableaux $T(L-\delta^{jr}+\delta^{is})$.
If  $(i,r)\neq (j,s)$ then  the coefficient of  $T(L-\delta^{jr}+\delta^{is})$ is zero similarly to the above case and, hence,  $[e_i,f_j]T(R)=0$ if $i\neq j$.

Suppose $i=j=k$.  The coefficient of $T(R-\delta^{ir}+\delta^{is})$ is zero if $r\neq s$.

By Corollary \ref{action of center1},
the coefficient of $T(R)$ is
\begin{align*}
&\lim_{v\rightarrow l}\left(\sum_{r=1}^{k}\sum_{s=1}^{k}f_{kr}(v)e_{ks}(v+\delta^{kt})
-\sum_{r=1}^{k}\sum_{s=1}^{k} e_{ks}(v)f_{kr}(v+\delta^{ks})\right)\\
&=\lim_{v\rightarrow R}h_{k}(v)
=h_{k}(R).
\end{align*}
Hence $(e_{i}f_{j}-f_{j}e_{i})T(R)=\delta_{ij}\frac{q^{\alpha_i}-q^{-\alpha_i}}{q-q^{-1}}T(R)$.

All other relations can be verified similarly. Thus $\mathcal{C}$ is realizable.

\end{proof}

\begin{remark}
The realizable sets of relations  for $\gl_n$ were all obtained in \cite{FRZ} (Theorem 4.22). Here we only deform the definition of a tableau satisfying a set of relations, i.e. replace $\mathbb{Z}$ by $\mathbb{Z}+\frac{1(q)}{2}$. As in the non quantum case, if we only consider irreducible modules, all realizable sets of relations are admissible. Thus the converse of Theorem \ref{sufficiency of admissible} holds.
\end{remark}

An effective method of constructing of realizable sets of relations was introduced  in \cite{FRZ}, called   {\it relations removal method} ({\it RR-method} for short). We will show that the same method can be applied to construct admissible sets of relations (and hence realizable subsets of $\mathcal{R}$ by Theorem \ref{condition for admissible},)  in the quantum case.

Let $\mathcal{C}$ be any realizable subset of $\mathcal{R}$ and  $T(L)$  a tableau satisfying $\mathcal{C}$. Fix $(k,i)\in \mathfrak{V}(\mathcal{C})$ and suppose that $T(L+m\delta^{ki})$ is a $\mathcal{C}$-realization for infinitely many choices of $m\in\mathbb{Z}$. Denote by $\widetilde{\mathcal{C}_{ki}}$ the set of relations obtained from $\mathcal{C}$ by removing all relations that involve $(k,i)$. We say that $\widetilde{\mathcal{C}}\subsetneq\mathcal{C}$ is obtained from $\mathcal{C}$ by the RR-method if it is obtained by a sequence of such  removing of relations for different indexes.

\begin{theorem}\label{RR}
Let $ \mathcal{C}_1$ be any realizable subset of $\mathcal{R}$. If $\mathcal{C}_2$ is obtained from $ \mathcal{C}_1$ by the RR-method  then $\mathcal{C}_2$ is realizable.
\end{theorem}

\begin{proof}
Analogous to the proof of Theorem 4.24 in \cite{FRZ}.
\end{proof}

We immediately obtain the following statement for generic modules which was shown in \cite{FRZ1}, Theorem 5.2 (cf. \cite{Zh}, Theorem 2).

\begin{corollary}
Let $T(L)$ be a generic Gelfand-Tsetlin tableau of height $n$.  Then $V_{\emptyset}(T(L))$ has a structure of a $U_q$-module with the action of the generators of $U_q$ given by the Gelfand-Tsetlin formulas (\ref{Gelfand-Tsetlin formulas}).
\end{corollary}

\begin{proof}
By Theorem \ref{condition for admissible} the set $\mathcal{S}$ is realizable and applying the RR-method to $\mathcal{S}$, after finitely many steps we can remove all the relations in $\mathcal{S}$, then  $\emptyset$ is realizable by Theorem \ref{RR}.
\end{proof}

We  call  $V_{\mathcal{C}}(T(L))$ {\it admissible Gelfand-Tsetlin module} associated with the admissible set of relations $\mathcal{C}$. Note that
$V_{\mathcal{C}}(T(L))$ is infinite dimensional if
$\mathcal{C}$  does not imply $\mathcal{S}$.

\section{Action of Gelfand-Tsetlin subalgebra}

From now on we will assume that $\mathcal{C}$ is an admissible  subset of $\mathcal{R}$ and consider the $U_q$-module
$V_{\mathcal{C}}(T(L))$. We will analyze the action of the Gelfand-Tsetlin subalgebra  $\Gamma_q$ on modules  $V_{\mathcal{C}}(T(L))$.

The action of $U_q$ on this irreducible module is given by the Gelfand-Tsetlin formulas (\ref{Gelfand-Tsetlin formulas}). Moreover, the action of $\Gamma_{q}$ is given by (\ref{eigenvalues of gamma_mk}). Then as in the non quantum case, we have the following Theorem:

\begin{theorem}\label{action of gamma}
For any admissible $\mathcal{C}$ the module $V_{\mathcal{C}}(T(L))$ is a Gelfand-Tsetlin module with diagonalisable action of
the generators of the Gelfand-Tsetlin subalgebra  given by the  formula \eqref{eigenvalues of gamma_mk}.
\end{theorem}

\begin{proof}
Essentially repeats the proof of Theorem 5.2  in \cite{FRZ}.
\end{proof}

\begin{remark}\label{same GT character adding 1(q)}

Note that for any $x\in 1(q)$ we have $\gamma_{mk}(L)=\gamma_{mk}(L+x\delta^{ij})$ for any $m,\ k$ and any $1\leq j\leq i\leq n$. In particular, the tableaux $T(L)$ and $T(L+x\delta^{ij})$ define the same Gelfand-Tsetlin character.

\end{remark}

\begin{remark}
Let $\mathcal{C}$ be any realizable set of relations and $T(L)$ satisfying  $\mathcal{C}$. Then for any ${x\in 1(q)}$ the tableau $T(L+x\delta^{ij})$ is also a $\mathcal{C}$-realization and the Gelfand-Tsetlin modules  $V_{\mathcal{C}}(T(L+x\delta^{ij}))$ and  $V_{\mathcal{C}}(T(L))$ are isomorphic (see Remark \ref{same GT character adding 1(q)}). On the other hand, when $\frac{x}{2}\notin 1(q)$, the two modules are not isomorphic.
In general, every admissible set of relations defines infinitely many nonisomorphic modules.
\end{remark}

Now we can show that $\Gamma_q$ separates basis tableaux in all constructed modules $V_{\mathcal{C}}(T(L))$ and hence, in their irreducible quotients.

\begin{proposition}\label{thm-mult}
For any $\sm\in \Sp \Gamma_q$ from the Gelfand-Tsetlin  support of $V_{\mathcal{C}}(T(L))$,  the Gelfand-Tsetlin multiplicity of $\sm$ is one.
\end{proposition}

\begin{proof}
The action of $\Gamma_q$ is given by the formulas (\ref{eigenvalues of gamma_mk}), and hence determined by the values
 of  symmetric polynomials on the entries of the rows of the tableaux. Given two Gelfand-Tsetlin tableaux $T(R_1)$ and $T(R_2)$ in $\mathcal{B}_{\mathcal{C}}(T(L))$, we have $c_{rs}(T(R_1))=c_{rs}(T(R_2))$ for any $1\leq s\leq r\leq n$ if and only if $R_1=\sigma(R_2)$ for some $\sigma\in G$.
But if $\sigma\neq 1$ then $T(R_1)$ and $T(R_2)$ have different order among the entries in $r$th row. Since both  $T(R_1)$ and $T(R_2)$ satisfy $\mathcal{C}$ and
$T(R_1)=T(R_2+z)$ for some $z\in \mathbb{Z}^{\frac{n(n-1)}{2}}$
  we come to  a contradiction.
\end{proof}

Therefore, we have an explicit basis of $V_{\mathcal{C}}(T(L))$, and of its irreducible quotients, parametrized by different Gelfand-Tsetlin tableaux with an explicit basis of the generators of $U_q$
and of $\Gamma_q$.

It was proved In \cite{FRZ1} that the irreducible module  containing generic tableau $T(R)$ has a basis of tableaux
$$\mathcal{I}(T(R))=\{T(S)\in\mathcal{B}(T(R)): \Omega^{+}(T(S))=\Omega^{+}(T(R))\},$$
where $$\Omega^{+}(T(S)):=\{(i,j,k)\ |\ s_{i,j}-s_{i-1,k}\in \frac{1(q)}{2}+\mathbb{Z}_{\geq 0}\}.$$

The following statement is a generalization of this result.
 Recall that $\mathcal{C}$ is a maximal set of relations satisfied by $T(L)$ if $T(L)$ is a $\mathcal{C}$-realization and  for any admissible set  of relations $\mathcal{C}'$ satisfied by $T(L)$, $\mathcal{C}$ implies $\mathcal{C}'$.

\begin{theorem}\label{thm-irr}
Admissible Gelfand-Tsetlin module $V_{\mathcal{C}}(T(L))$ is irreducible if and only if $\mathcal{C}$ is a maximal set of relations satisfied by $T(L)$.
\end{theorem}
\begin{proof}
Let $T(R)$ be any tableau in $\mathcal{B}_{\mathcal{C}}(T(L))$ and $\mathcal{C}$  a maximal set of relations satisfied by $T(L)$. One can show easily that $U_q T(R)\subseteq V_{\mathcal{C}}(T(L))$.
If $T(R')$ is another tableau in $V_{\mathcal{C}}(T(L))$ then there exist $\{(i_s,j_s)\}\subseteq \{(i,j)\ |\ 1\leq j\leq i\leq n-1)\}$ and $t$ such that,
for any $k\leq t$, $T(R_k)=T(R+\sum_{s=1}^{k}p_s\delta^{i_s,j_s})\in \mathcal{B}_{\mathcal{C}}(T(L))$, where $p_i\in\{1, -1\}$ and
$T(R_0)=T(R)$, $T(R_t)=T(R')$. It is sufficient to show that we can obtain $T(R_s)$ from $T(R_{s-1})$.
If $p_i=1$ (resp. $p_i=-1$) then acting by $e_{i_s}$ (resp. $f_{i_s}$) on $T(R_{s-1})$ the coefficient of $T(R_s)$ in the image is not zero.
By Theorem \ref{action of gamma}, there exists an element in $\Gamma_q$ which annihilates all other tableaux except $T(R_s)$.
We conclude that $V_{\mathcal{C}}(T(L))\subseteq  U_qT(R)$.

Conversely, assume $T(L)$ satisfies $\mathcal{C}$ and $\mathcal{C}$ is not maximal.  Let $\mathcal{C}'$ be the maximal set of relations satisfied by $T(L)$. Then $V_{\mathcal{C}'}(T(L))$ is a subquotient of $V_{\mathcal{C}}(T(L))$ and $V_{\mathcal{C}'}(T(L))\neq V_{\mathcal{C}}(T(L))$. It contradicts the irreducibility of $V_{\mathcal{C}}(T(L))$.
\end{proof}

We conclude with an example of the  family of highest weight modules that can be realized as $V_{\mathcal{C}}(T(L))$ for some admissible set of relations $\mathcal{C}$.

\begin{proposition}\label{highest weight module}
Set $\lambda=(\lambda_1,\ldots ,\lambda_n)$. The irreducible highest weight module $L(\lambda)$ is admissible Gelfand-Tsetlin module if  $\lambda_{i}-\lambda_j\notin \mathbb{Z}$ or $\lambda_{i}-\lambda_j> i-j$ for any $1\leq i<j\leq n-1$.
\end{proposition}

\begin{proof}
Let $T(L)$ be a tableau such that $l_{ij}=\lambda_i$, $\mathcal{C}$ be the maximal set of relations satisfied by $T(L)$. Then
$\mathcal{C}$ is admissible and $V_{\mathcal{C}}(T(L))$ is irreducible admissible module.
Moreover,  $T(L)$ is a  highest weight vector and
 $V_{\mathcal{C}}(T(L))$ is isomorphic to $L(\lambda)$.
\end{proof}


\end{document}